\documentclass[a4,10pt]{article}
\usepackage{probpack_eng}

\title{The law of the iterated logarithm for a class of transient random walk in random environment}
\author{Naoki KUBOTA\footnote{Department of Mathematics, e-mail: kubota@grad.math.cst.nihon-u.ac.jp}\\
        Graduate School of Science and Technology, Nihon University}
\begin{document}
\maketitle

\begin{abstract}
There is a condition (T'), such that it is the necessary condition that a random walk in random environment is ballistic. 
Under this condition, we show the law of the iterated logarithm for a random walk in random environment. 
\end{abstract}

\section{Introduction}
Random walks in random environment constitute one of the basic models of random motions in random media. 
Their asymptotic behavior has been extensively investigated in the one-dimensional case, see for details in \cite{MR1420619} or \cite{MR2071631}. 
On the other hand, the multi-dimensional case has been the object of comparatively few works and remains altogether poorly understood. 
Recent advances have mainly been concerned with the ballistic situation where the walk has a non-degenerate asymptotic velocity. 
The recent article \cite{MR1902189} introduced the condition (T'), which implies that a random walk in random environment is ballistic for a dimension bigger than one. 
Under this condition, we consider the asymptotic behavior of a random walks in random environment. 

We begin to explain the setting. 
A random walk in random environment (RWRE) on $\Z^d \,(d \geq 1)$ is defined as follows: 
let $\mathcal{E}_d:=\{ e \in \Z^d;|e|=1 \}$, and we consider
\begin{align*}
 \mathcal{M}_d:= \biggl\{ p : \mathcal{E}_d \to [0,1];\sum_{|e|=1} p(e)=1 \biggr\}.
\end{align*}
For a given probability measure $\mu$ on $\mathcal{M}_d$, we set
\begin{align*}
 \Omega:=(\mathcal{M}_d)^{\Z^d},\quad
 \P:=\mu^{\Z^d},
\end{align*}
and consider the probability space $(\Omega,\mathcal{G},\P)$, where $\mathcal{G}$ is the $\sigma$-field generated by cylinder sets on $\Omega$. 
We call an element $\omega=(\omega(x,\cdot))_{x \in \Z^d}$ of $\Omega$ the \textit{random environment}. 
Through this paper, we always assume that $\P$ is \textit{uniformly elliptic}, \textit{i.e.} There exists a constant $\kappa>0$ such that $\P(\omega(0,e) \geq \kappa)=1, |e|=1$D
For each environment $\omega \in \Omega$, the \textit{random walk in random environment} $\omega$ is a time-homogeneous Markov chain $((X_n)_{n=0}^\infty,(P_\omega^x)_{x \in \Z^d})$ with the state space $\Z^d$, under which: 
\begin{align*}
 &P_\omega^x(X_0=x)=1,\\
 &P_\omega^x(X_{n+1}-X_n=e|\mathcal{F}_n)=\omega(X_n,e),\qquad
  |e|=1, x \in \Z^d,
\end{align*}
where $\mathcal{F}_n:=\sigma(X_0,\dots,X_n), n \geq 0$. 
We call $P_\omega^x$ the \textit{quenched law}. 
We set $\mathcal{F}=\sigma(\cup_{n=0}^\infty \mathcal{F}_n)$ and define the \textit{annealed law} $P^0$ on $(\Omega \times (\Z^d)^{\Z_{\geq 0}},\mathcal{G} \times \mathcal{F})$ by
\begin{align*}
 P^0(G \times F):=\int_{G} P_\omega^0(F) \,\P(d \omega),\qquad
 G \in \mathcal{G}, F \in \mathcal{F}.
\end{align*}
We denote by $\E, E_\omega^x,$ and $E^0$ expectations of $\P, P_\omega^x,$ and $P^0$, respectively. 

A regeneration structure is important to catch the future of ballistic RWRE. 
We shall prepare some notations to consider a regeneration structure. 
For $\ell \in S^{d-1}:=\{ w \in \R^d;|w|=1 \}$, we consider stopping times 
\begin{align*}
 &D:=\inf \{ n \geq 0 ; X_n \cdot \ell < X_0 \cdot \ell \},\\
 &T_{\geq L}:=\inf \{ n \geq 0;X_n \cdot \ell \geq L \},\qquad
 L \in \R.
\end{align*}
We will define the regeneration time, which is, roughly speaking, the first time where $X_n \cdot \ell$ goes by its previous local maxima and never goes below this level from then on. 
More precisely, the regeneration time is defined as follows: 
We define random times and associated random variables 
\begin{align*}
 &S_{k+1}:=T_{\geq M_k+1},\quad S_0:=0,\\
 &R_{k+1}:=D \circ \theta_{S_{k+1}}+S_{k+1},\\
 &M_{k+1}:=\sup \{ X_m \cdot \ell ; 0 \leq m \leq R_{k+1} \},\quad
  M_0:=X_0 \cdot \ell,\qquad k \geq 0,\\
 &K:=\inf \{ j \geq 1 ; S_j<\infty, R_j=\infty \}, 
\end{align*}
where $\theta$ is the canonical shift on the space of trajectories. 
In addition, the \textit{regeneration time} is $\tau_1:=S_K$ with the convention $S_\infty=\infty$. 

We are interested in the asymptotic behavior of the ballistic RWRE. 
Sznitman \cite{MR1902189} introduced conditions $\textrm{(T)}_\gamma|\ell$ and $\textrm{(T')}|\ell$ to analyze the ballistic RWRE; 
for a direction $\ell \in S^{d-1}$ and $\gamma \in (0,1]$, we say ``the $\textit{condition (T)}_\gamma|\ell$ holds'' if the following conditions hold: 
\begin{enumerate}
 \item $P^0(\lim_{n \to \infty} X_n \cdot \ell=\infty)=1$,
 \item for some constant $c>0$, 
       \begin{align}
        E^0[\exp (c X_\gamma^*)]<\infty, 
        \label{eq:T}
       \end{align}
       where $X_\gamma^*:=\sup \{ |X_k|^\gamma;0 \leq k \leq \tau_1 \}$. 
\end{enumerate}
And we say ``the $\textit{condition (T')}|\ell$ holds'' if $\textrm{(T)}_\gamma|\ell$ is approved for every $0<\gamma<1$. 
After this, we assume that a dimension $d$ is bigger than one and that $\textrm{(T')}|\ell$ holds for some $\ell \in S^{d-1}$. 
Under the $\textrm{condition (T')}|\ell$, it is already known that the law of large numbers and the central limit theorem hold, see for details in \cite{MR1902189} or \cite{MR2198849}: 
$P^0$-almost surely, 
\begin{align}
 \lim_{n \to \infty} \frac{X_n}{n}=v:=\frac{E^0[X_{\tau_1} | D=\infty]}{E^0[\tau_1 | D=\infty]},\quad v \cdot \ell >0
 \label{eq:lil3}
\end{align}
holds, and $B_\cdot^n:=\frac{1}{\sqrt{n}} (X_{[\cdot n]}-[\cdot n] v)$ converges in law to a Brownian motion with non-degenerate covariance matrix. 

We would like to discuss the law of the iterated logarithm for RWRE under same settings to the above. 
The main result of this paper is the following: 

\begin{thm} \label{thm:getdoc_3eng_1}
For any $u \in S^{d-1}$, 
\begin{align}
 &\limsup_{n \to \infty} \frac{(X_n-nv) \cdot u}{E^0[\tau_1 | D=\infty ]^{-\half} (2c_u n \log \log n^\half)^\half}
  = 1,\label{eq:getdoc_3eng_1}\\
 &\liminf_{n \to \infty} \frac{(X_n-nv) \cdot u}{E^0[\tau_1 | D=\infty ]^{-\half} (2c_u n \log \log n^\half)^\half}
  = -1, \label{eq:getdoc_3eng_2}
\end{align}
$P^0$-almost surely, where the constant $c_u$ is defined by
\begin{align*}
 c_u
 := \var_{P^0(\cdot | D=\infty)}((X_{\tau_1}-\tau_1 v) \cdot u)
 = E^0[((X_{\tau_1}-\tau_1 v) \cdot u)^2 | D=\infty].
\end{align*}
\end{thm}

\section{Proof of the main theorem}
We get ready some notations for the proof of Theorem \ref{thm:getdoc_3eng_1}. 
We define \textit{regeneration times} $(\tau_k)_{k=0}^\infty$ as follows: 
\begin{align*}
 &\tau_0 :=0,\\
 &\tau_{k+1}:= \tau_k(X_\cdot)+\tau_1(X_{\tau_k+\cdot}-X_{\tau_k}),\qquad k \geq 1. 
\end{align*}
In addition, we define the random sequence $(k_n)_{n=0}^\infty$ by $\tau_{k_n} \leq n <\tau_{k_n+1}$ for $n \geq 0$. 
Note that 
\begin{align*}
 ((X_{\tau_1 \wedge \cdot}), \tau_1),\quad
 ((X_{(\tau_k+\cdot) \wedge \tau_{k+1}}-X_{\tau_k}),\tau_{k+1}-\tau_k),\qquad k \geq 1
\end{align*}
are independent variables under $P^0$, and furthermore, the law of 
\begin{align*}
 ((X_{(\tau_k+\cdot) \wedge \tau_{k+1}}-X_{\tau_k}),\tau_{k+1}-\tau_k),\qquad k \geq 1
\end{align*}
coincides with that of $((X_{\tau_1 \wedge \cdot}), \tau_1)$ under $P^0(\cdot |D=\infty)$ from Theorem 1.4 of \cite{MR1742891}. 
We also set
\begin{align*}
 Z_k^u:=((X_{\tau_k}-X_{\tau_{k-1}})-(\tau_k-\tau_{k-1})v) \cdot u,
 \qquad k \geq 1, u \in S^{d-1}, 
\end{align*}
where $v$ is the same vector as in the description of the law of large numbers \eqref{eq:lil3}. 

Let us start the proof of Theorem \ref{thm:getdoc_3eng_1}. 
Let us prove \eqref{eq:getdoc_3eng_1} only, because we can obtain \eqref{eq:getdoc_3eng_2} for $u \in S^{d-1}$ by applying \eqref{eq:getdoc_3eng_1} for the opposite vector $-u$. 
We shall decompose the quantity appearing in the left hand side of \eqref{eq:getdoc_3eng_1} as follows: 
\begin{align}
 &\frac{(X_n-nv) \cdot u}{E^0[\tau_1 | D=\infty ]^{-\half} (2c_u n \log \log n^\half)^\half} \nonumber\\
 &= \frac{\sum_{j=1}^{k_n} Z_j^u}{E^0[\tau_1 | D=\infty ]^{-\half} \varphi(c_u k_n)}
    \biggl( \frac{k_n \log \log (c_u k_n)^\half}{n \log \log n^\half} \biggr)^\half \nonumber \\
 &\quad
    +\frac{(X_n-X_{\tau_{k_n}}) \cdot u}{E^0[\tau_1 | D=\infty ]^{-\half} (2c_u n \log \log n^\half)^\half} \nonumber\\
 &\quad
    +\frac{(n-\tau_{k_n}) (v \cdot u)}{E^0[\tau_1 | D=\infty ]^{-\half} (2c_u n \log \log n^\half)^\half}.
 \label{eq:last}
\end{align}

We hope that the first term in the right hand side of the above expression is the main term to conclude the goal \eqref{eq:getdoc_3eng_1}. 
Now, to check that this conjecture is true, let us show that the second term and the third term in the right hand side of the above expression are error terms. 

\begin{lem} \label{lem:lil_2}
We have the following $P^0$-almost surely: 
\begin{align}
 &\lim_{n \to \infty} \frac{(X_n-X_{\tau_{k_n}}) \cdot u}{E^0[\tau_1 | D=\infty ]^{-\half} (2c_u n \log \log n^\half)^\half}
  =0, \label{eq:I_2} \\
 &\lim_{n \to \infty} \frac{(n-\tau_{k_n}) (v \cdot u)}{E^0[\tau_1 | D=\infty ]^{-\half} (2c_u n \log \log n^\half)^\half}
  =0. \label{eq:I_3}
\end{align}
\end{lem}
\begin{proof}
We start to show \eqref{eq:I_2}. 
Let us fix a $\gamma \in (0,1)$. 
For any $\epsilon >0$, we obtain 
\begin{align*}
 &P^0 \biggl( \frac{|(X_n-X_{\tau_{k_n}}) \cdot u|}{(2c_u n \log \log n^\half)^\half} >\epsilon \biggr)\\
 &\leq \exp \{ -c \epsilon (2c_u n \log \log n^\half)^\frac{\gamma}{2} \} E^0[\exp (c X_\gamma^* \circ \theta_{\tau_{k_n}})]
\end{align*}
by Chebyshev's inequality, where the constant $c>0$ as in the definition of (T)$_\gamma|\ell$, see \eqref{eq:T}.  
Since we obtain  
\begin{align*}
 E^0[\exp (c X_\gamma^* \circ \theta_{\tau_{k_n}})]
 \leq 2 E^0[\exp (c X_\gamma^*)] P^0(D=\infty)^{-1} n, 
\end{align*}
we now get 
\begin{align}
 &P^0 \biggl( \frac{|(X_n-X_{\tau_{k_n}}) \cdot u|}{(2c_u n \log \log n^\half)^\half} >\epsilon \biggr) \nonumber \\
 &\leq 2 E^0[\exp (c X_\gamma^*)] P^0(D=\infty)^{-1} n
       \exp \{ -c \epsilon (2c_u n \log \log n^\half)^\frac{\gamma}{2} \}.
 \label{eq:getdoc_3eng_8}
\end{align}
Because the left hand side of \eqref{eq:getdoc_3eng_8} is summable in $n$, the Borel-Cantelli lemma implies that we conclude \eqref{eq:I_2}. 

Similarly to the above, we can find that \eqref{eq:I_3} holds, because $\tau_1$ possesses all moments, see Proposition 3.1 of \cite{MR1902189} or Theorem 4.2 of \cite{MR2198849}, and we are appreciable as follows: for $\epsilon >0$, 
\begin{align*}
 &P^0 \biggl( \frac{n-\tau_{k_n}}{(2c_u n \log \log n^\half)^\half} >\epsilon \biggr)\\
 &\leq 2n P^0(D=\infty)^{-1} (\epsilon (2c_u n \log \log n^\half)^\half)^{-6} E^0[\tau_1^6]. 
\end{align*}
Therefore we finished to prove Lemma \ref{lem:lil_2}. 
\end{proof}

We find that the first term in the right hand side of \eqref{eq:last} is the main term to conclude the goal \eqref{eq:getdoc_3eng_1} from Lemma \ref{lem:lil_2}. 
Finally, we should obtain an appropriate estimate of this term in order to conclude the goal \eqref{eq:getdoc_3eng_1}. 
So let us show the following lemma: 

\begin{lem} \label{lem:lil_1}
We have $P^0$-almost surely 
\begin{align}
 \limsup_{n \to \infty} \frac{\sum_{j=1}^{k_n} Z_j^u}{E^0[\tau_1 | D=\infty ]^{-\half} \varphi(c_u k_n)}
    \biggl( \frac{k_n \log \log (c_u k_n)^\half}{n \log \log n^\half} \biggr)^\half
 =1.
 \label{eq:I_1}
\end{align}
\end{lem}
\begin{proof}
It is clear that $E^0[Z_k^u]=0$ holds for $k \geq 1$ , in fact
\begin{align*}
 E^0[\tau_1|D=\infty]^{-1} E^0[Z_k^u]
 = E^0[\tau_1|D=\infty]^{-1} E^0[X_{\tau_1}-\tau_1 v|D=\infty] \cdot u
 = 0.
\end{align*}
We define 
\begin{align*}
 &s_k^2:=\sum_{k=1}^n \var_{P^0}(Z_n^u)=E^0[(X_{\tau_1} \cdot u)^2]+(k-1)c_u,\\
 &\Gamma_k:=\sum_{k=1}^n E^0[|Z_k^u|^3]=E^0[|X_{\tau_1} \cdot u|^3]+(k-1) \hat{c}_u,
\end{align*}
where $\hat{c}_u:=E^0[|X_{\tau_1} \cdot u|^3 | D=\infty]$. 
For $0<\epsilon <1$, we get 
\begin{align*}
 &\frac{\Gamma_k}{s_k^3} (\log s_k)^{1+\epsilon}\\
 &= \frac{E^0[|X_{\tau_1} \cdot u|^3]+(k-1) \hat{c}_u}{k \hat{c}_u}
    \biggl( \frac{k c_u}{E^0[(X_{\tau_1} \cdot u)^2]+(k-1)c_u} \biggr)^{\frac{3}{2}}\\
 &\quad \times
    \half \hat{c}_u c_u^{-\frac{3}{2}} k^{-\half}
    \biggl\{ \log \biggl( \frac{E^0[(X_{\tau_1} \cdot u)^2]+(k-1)c_u}{k c_u} \biggr)
            +\log (kc_u) \biggr\}^{1+\epsilon}\\
 &\xrightarrow{k \to \infty} 0.
\end{align*}
Combining the above with Theorem 7.5.1 of \cite{MR1796326}, we have the law of the iterated logarithm for $(Z_k^u)_{k=1}^\infty$: 
\begin{align*}
 \limsup_{k \to \infty} \frac{\sum_{j=1}^k Z_j^u}{\varphi(s_k^2)}
 = 1,\qquad P^0 \hyphen \as, 
\end{align*}
where $\varphi(x):=(2x \log \log x^\half)^\half$. 
By $\lim_{k \to \infty} \frac{s_k^2}{c_u k}=1$, we get 
\begin{align}
 \limsup_{k \to \infty} \frac{\sum_{j=1}^k Z_j^u}{\varphi(c_u k)}
 = \limsup_{k \to \infty} \frac{\sum_{j=1}^k Z_j^u}{\varphi(s_k^2)}
                          \biggl( \frac{s_k^2 \log \log s_k}{c_u k \log \log (c_u k)^\half} \biggr)^\half
 = 1,\qquad P^0 \hyphen \as
 \label{eq:getdoc_3eng_3}
\end{align}
To obtain the conclusion \eqref{eq:I_1}, it is sufficient to show
\begin{align}
 \limsup_{n \to \infty} \frac{\sum_{j=1}^{k_n} Z_j^u}{\varphi(c_u k_n)}=1,\qquad P^0 \hyphen \as
 \label{eq:getdoc_3eng_7}
\end{align}
because we have 
\begin{align*}
 \lim_{n \to \infty} \frac{k_n}{n}=E^0[\tau_1|D=\infty]^{-1},\qquad P^0 \hyphen \as
\end{align*}
from Proposition 2.1 of \cite{MR1742891}. 

Let us show \eqref{eq:getdoc_3eng_7}. 
For a fixed $m \geq 0$, noting that $k_m \leq n$ holds for all $n \geq m$ by the definition of $(k_n)_{n=0}^\infty$, we have
\begin{align*}
 \sup_{i \geq n} \frac{\sum_{j=1}^i Z_j^u}{\varphi(c_u i)}
 \leq \sup_{i \geq k_m} \frac{\sum_{j=1}^i Z_j^u}{\varphi(c_u i)}
 = \sup_{i \geq m} \frac{\sum_{j=1}^{k_i} Z_j^u}{\varphi(c_u k_i)}. 
\end{align*}
Letting $n \to \infty$ and then $m \to \infty$, By \eqref{eq:getdoc_3eng_3}, we get 
\begin{align}
 1
 \leq \limsup_{n \to \infty} \frac{\sum_{j=1}^{k_n} Z_j^u}{\varphi(c_u k_n)}. 
 \label{eq:lil1}
\end{align}
On the other hand, for a fixed $m \geq 0$, we have 
\begin{align*}
 \sup_{i \geq n} \frac{\sum_{j=1}^{k_i} Z_j^u}{\varphi(c_u k_i)}
 \leq \sup_{i \geq m} \frac{\sum_{j=1}^i Z_j^u}{\varphi(c_u i)}
\end{align*}
for a large enough $n$. 
Similarly to the above, letting $n \to \infty$ and then $m \to \infty$, 
we can obtain
\begin{align}
 \limsup_{n \to \infty} \frac{\sum_{j=1}^{k_n} Z_j^u}{\varphi(c_u k_n)}
 \leq 1. 
 \label{eq:lil2}
\end{align}
We conclude \eqref{eq:getdoc_3eng_7} from \eqref{eq:lil1} and \eqref{eq:lil2}. 
\end{proof}

Therefore, the equality \eqref{eq:getdoc_3eng_1} follows immediately from the equality \eqref{eq:last}, Lemma \ref{lem:lil_2} and Lemma \ref{lem:lil_1}. 

\paragraph{Acknowledgment.}
The author thanks Professor Shigenori Matsumoto and Professor Takao Nishikawa for many useful discussions and important comments on this paper. 

\bibliographystyle{plain}

\begin{thebibliography}{1}

\bibitem{MR1796326}
Kai~Lai Chung.
\newblock {\em A course in probability theory}.
\newblock Academic Press Inc., San Diego, CA, third edition, 2001.

\bibitem{MR1420619}
Barry~D. Hughes.
\newblock {\em Random walks and random environments. {V}ol. 2}.
\newblock Oxford Science Publications. The Clarendon Press Oxford University
  Press, New York, 1996.
\newblock Random environments.

\bibitem{MR1902189}
Alain-Sol Sznitman.
\newblock An effective criterion for ballistic behavior of random walks in
  random environment.
\newblock {\em Probab. Theory Related Fields}, Vol. 122, No.~4, pp. 509--544,
  2002.

\bibitem{MR2198849}
Alain-Sol Sznitman.
\newblock Topics in random walks in random environment.
\newblock In {\em School and {C}onference on {P}robability {T}heory}, ICTP
  Lect. Notes, XVII, pp. 203--266 (electronic). Abdus Salam Int. Cent. Theoret.
  Phys., Trieste, 2004.

\bibitem{MR1742891}
Alain-Sol Sznitman and Martin Zerner.
\newblock A law of large numbers for random walks in random environment.
\newblock {\em Ann. Probab.}, Vol.~27, No.~4, pp. 1851--1869, 1999.

\bibitem{MR2071631}
Ofer Zeitouni.
\newblock Random walks in random environment.
\newblock In {\em Lectures on probability theory and statistics}, Vol. 1837 of
  {\em Lecture Notes in Math.}, pp. 189--312. Springer, Berlin, 2004.

\end{thebibliography}

\end{document}